\documentclass[11pt]{amsart}
\usepackage{amsfonts}
\usepackage{mathrsfs}
\usepackage{bbm}
\usepackage{amssymb}
\usepackage{indentfirst}
\usepackage{latexsym,amsfonts,amssymb,amsmath}

\newtheorem{theorem}{Theorem}[section]
\newtheorem{lemma}[theorem]{Lemma}
\newtheorem{corollary}[theorem]{Corollary}
\newtheorem{question}[theorem]{Question}
\newtheorem{problem}[theorem]{Problem}
\theoremstyle{definition}
\newtheorem{definition}[theorem]{Definition}
\newtheorem{proposition}[theorem]{Proposition}
\newtheorem{example}[theorem]{Example}
\theoremstyle{remark}
\newtheorem{remark}[theorem]{Remark}
\newcommand{\cont}{\mathfrak{c}}
\newcommand{\supp}{\mathrm{supp}}
\DeclareMathSymbol{\res}{\mathord}{AMSa}{"16}
\def\hull#1{\langle#1\rangle}

\begin{document}

\title{Simply $sm$-factorizable (para)topological groups and their completions}

\author{Li-Hong Xie*}
\address{(L.H. Xie) School of Mathematics and Computational Science, Wuyi University,
Jiangmen 529020, P.R. China} \email{yunli198282@126.com; xielihong2011@aliyun.com}
\thanks{*The research is supported by NSFC (Nos. 11601393; 11861018)}

\author{Mikhail Tkachenko**}
\address{(M. Tkachenko) Departamento de Matem\'aticas,
Universidad Aut\'onoma Metropolitana,
 Av. San Rafael Atlixco 186,
Col. Vicentina, Iztapalapa, C.P. 09340, M\'exico City, Mexico}
\email{mich@xanum.uam.mx}
\thanks{** Corresponding author}
\date{February 9, 2020}

\begin{abstract}
Let us call a (para)topological group \emph{strongly submetrizable} if it admits a coarser
separable metrizable (para)topological group topology. We present a characterization of simply
$sm$-factorizable (para)topo\-logical groups by means of continuous real-valued functions. We show
that a (para)topo\-logical group $G$ is a simply $sm$-factorizable if and only if for each continuous
function $f\colon G\to \mathbb{R}$, one can find a continuous homomorphism $\varphi$ of $G$ onto a
strongly submetrizable (para)topological group $H$ and a continuous function $g\colon H\to \mathbb{R}$
such that $f=g\circ\varphi$. This characterization is applied for the study of completions of simply $sm$-factorizable topological groups. We prove that the equalities $\mu{G}=\varrho_\omega{G}=\upsilon{G}$ hold
for each Hausdorff simply $sm$-factorizable topological group $G$. This result gives a positive answer to a question posed by Arhangel'skii and Tkachenko in 2018. Also, we consider realcompactifications of simply $sm$-factorizable paratopological groups. It is proved, among other results, that the realcompactification,
$\upsilon{G}$, and the Dieudonn\'e completion, $\mu{G}$, of a regular simply $sm$-factorizable paratopological group $G$ coincide and that $\upsilon{G}$ admits the natural structure of paratopological group containing $G$ as a dense subgroup and, furthermore, $\upsilon{G}$ is also simply $sm$-factorizable. Some results in [\emph{Completions of paratopological groups, Monatsh. Math. \textbf{183} (2017), 699--721}] are improved or generalized.
\end{abstract}

\maketitle

\section{Introduction}

A \emph{paratopological group} $G$ is a group $G$ with a topology such that the
multiplication mapping of $G \times G$ to $G$ associating $xy$ to arbitrary $x, y\in G$
is jointly continuous. A paratopological group $G$ is called a \emph{topological group} if the
inversion on $G$ is continuous.

Slightly reformulating the celebrated theorem of Comfort and Ross \cite[Theorem~1.2]{CR},
we can say that the pseudocompact topological groups are exactly the dense $C$-embedded
subgroups of compact topological groups. In particular, the Stone-\v{C}ech compactification,
$\beta{G}$, the Hewitt-Nachbin completion, $\upsilon{G}$, and the Ra\u{\i}kov completion,
$\varrho{G}$, of a pseudocompact topological group $G$ coincide.  Hence the Hewitt-Nachbin
completion of the group $G$ is again a topological group containing $G$ as a dense
$C$-embedded subgroup. Recently, with the idea to study connections between the properties
of a Tychonoff space $X$ and its dense $C$-embedded subspace $Y$ homeomorphic to
a topological group, the authors of \cite{AT1} introduced the new notions of $sm$-factorizable,
densely $sm$-factorizable and simply $sm$-factorizable (para)topological groups as follows:

\begin{definition}[See Definition~5.11 in \cite{AT1}]
A (para)topological group $G$ is called \textit{$sm$-factorizable} if for each co-zero set
$U$ in $G$, there exists a continuous homomorphism $\pi$ of $G$ onto a separable metrizable
(para)topological group $H$ such that the set $\pi(U$) is open in $H$ and $\pi^{-1}(\pi(U)) = U$.
Replacing the assumption that $\pi(U)$ is open by the requirement that $\pi(U)$ is dense in some
open subset of $H$, we obtain the definition of a \textit{densely $sm$-factorizable} (para)topological
group. Removing the assumption that $\pi(U)$ is open, we obtain the definition of a \textit{simply
$sm$-factorizable} (para)topological group.
\end{definition}

It is shown in \cite{AT1} that the following implications are valid (and none of these implications can
be inverted, see \cite[Example~5.12, Proposition~5.13]{AT1}):
$$
sm\text{-factorizability}\, \Rightarrow\, \text{dense~} sm\text{-factorizability}\, \Rightarrow\, \text{simple~}
sm\text{-factorizability}.
$$

Recall that a (para)topological group $G$ is called {\it $\mathbb{R}$-factorizable} if for every continuous
real-valued function on $G$ can be factorized through a continuous homomorphism onto a separable metrizable (para)topological group. Arhangel'skii and Tkachenko obtained that a (para)topological group
$G$ is $\mathbb{R}$-factorizable if and only if $G$ is $sm$-factorizable \cite[Theorem~5.9]{AT1}
(in view of the proof of \cite[Theorem~5.9]{AT1}, it is worth noting that it need not any separation axiom
on $G$). However, there is a Hausdorff $\omega$-narrow and simply $sm$-factorizable Abelian topological group is not $\mathbb{R}$-factorizable \cite[Corollary 5.20]{AT1}. In \cite[Corollary~5.24]{AT1} it is shown that subgroups of Lindel\"{o}f topological groups need not be simply $sm$-factorizable. Since subgroups of Lindel\"{o}f topological groups are $\omega$-narrow, $\omega$-narrow topological groups can fail to be simply $sm$-factorizable. Also, there exists a densely $sm$-factorizable topological Abelian group $G$ such that $ib(G) = 2^\omega$ \cite[Example~5.12]{AT1}, so densely $sm$-factorizable topological groups can fail to be $\omega$-narrow.
Now we summarize the relations between $\mathbb{R}$-factorizable, $sm$-factorizable, densely
$sm$-factorizable, simply $sm$-factorizable and $\omega$-narrow (para)topological groups as follows:
\begin{enumerate}
\item[(1)] $\mathbb{R}\text{-factorizability} \Leftrightarrow sm\text{-factorizability}$;
\item[(2)] $sm\text{-factorizability} \Rightarrow \text{dense~} sm\text{-factorizability} \Rightarrow
               \text{simple~}  sm\text{-factorizability};$
\item[(3)] regularity $\&$ $sm$-factorizability $\Rightarrow$ total $\omega$-narrowness (in paratopo\-logical
               groups);
\item[(4)] regularity $\&$ $\omega$-narrowness  $\not \Rightarrow$ simple $sm$-factorizability (in topological
               groups);
\item[(5)] regularity $\&$ dense $sm$-factorizability $\not\Rightarrow$ $\omega$-narrowness (in topological
               groups).
\end{enumerate}

 We denote the Dieudonn\'{e} completion and Hewitt-Nachbin realcompactification of a completely
 regular space $X$ by $\mu X$ and $\upsilon X$, respectively. Recall that a topological group $G$
 is called a {\it $PT$-group} if the group operations in $G$ can be continuously extended to
 Dieudonn\'{e} completion $\mu G$. Since every $\mathbb{R}$-factorizable topological group
 is a $PT$-group \cite[Corollary~8.3.7]{AT}, the authors of \cite{AT1} posed the following
 question:

 \begin{problem}[See Problem~7.6 in \cite{AT1}]\label{Q1.1}
Is every simply $sm$-factorizable topological group a $PT$-group? What
about densely $sm$-factorizable topological groups?
\end{problem}

Continuous open homomorphic images of $\mathbb{R}$-factorizable topological groups are
$\mathbb{R}$-factorizable \cite[Theorem~8.4.2]{AT}, but it is unknown whether continuous
homomorphisms preserve $\mathbb{R}$-factorizable topological groups \cite[Open~Problem~8.4.1]{AT}.
A weaker version of this problem is given below:

\begin{problem}[See Problem~7.8 in \cite{AT1}]\label{Q1.2}
Is every continuous homomorphic image of an $\mathbb{R}$-factorizable topological group
$G$ simply $sm$-factorizable?
\end{problem}

It is known, however, that continuous homomorphic images of simply $sm$-factorizable topological
groups need not be simply $sm$-factorizable \cite{AT1}. This makes it natural to rise the following
question:

\begin{problem}[See Problem~7.9 in \cite{AT1}]\label{Q1.3}
Is every quotient group of a simply $sm$-factorizable topological group
$G$ simply $sm$-factorizable? What if, additionally, $G$ is $\omega$-narrow?
\end{problem}

In Section~\ref{Sec:2}, we present a characterization of simply $sm$-factorizable (para)topo\-logical
groups in terms of continuous homomorphisms onto strongly submetrizable (para)topological groups.
Applying this characterization we give a positive answer to Problem~\ref{Q1.1} and partially answer Problems~\ref{Q1.2} and~\ref{Q1.3}.

The article is organized as follows. In Section~\ref{Sec:2}, we characterize simply $sm$-factorizable (para)topological groups. We establish the following facts:
(1) A (para)topological group $G$ is simply $sm$-factorizable if and only if for each continuous function
$f\colon G\to \mathbb{R}$, one can find a continuous homomorphism $\varphi$ of $G$ onto a strongly
submetrizable (para)topo\-logical group $H$ and a continuous function $g\colon H\to \mathbb{R}$
such that $f=g\circ\varphi$ (see Theorem~\ref{th1}); (2) an $\omega$-narrow topological group $G$ is
simply $sm$-factorizable if and only if for any continuous function $f\colon G\to \mathbb{R}$, there
is an invariant admissible subgroup $N_f$ of $G$ such that $f$ is constant on $gN_f$ for each $g\in G$
(see Theorem~\ref{Th3}). Section~\ref{Sec:3} contains a positive answer to a question posed by Arhangel'skii and Tkachenko \cite[Problem~7.8]{AT1}. We show that the equalities $\mu{G}=\varrho_\omega{G}=\upsilon{G}$ hold
for every Hausdorff simply $sm$-factorizable topological group $G$ and, therefore, $G$ is completion
friendly (see Theorem~\ref{Th3.2}).

In Section~\ref{Sec:4}, we study completions of simply $sm$-factorizable paratopological groups.
We establish the following:
(1) If $G$ is a regular simply $sm$-factorizable paratopological group, then the realcompactification
$\upsilon{G}$ of the space $G$ admits a natural structure of paratopological group containing $G$
as a dense subgroup and $\upsilon{G}$ is also simply $sm$-factorizable (Theorem~\ref{Th});
(2) If $G$ is a regular paratopological group such that the topological group $G^\ast$ associated
to $G$ is $\omega$-narrow and simply $sm$-factorizable, then the realcompactification $\upsilon{G}$
of $G$ admits a natural structure of paratopological group containing $G$ as a dense subgroup and
the equality $\upsilon{G}=\mu{G}$ holds (Theorem~\ref{Th3.14}).

Our results are accompanied with several examples and short discussions that outline the limits
for their generalizations.

\smallskip
We do not impose any separation restrictions on spaces and (para)topological groups unless
the separation axioms are stated explicitly. A space $X$ satisfies the $T_3$-separation axiom
if for any point $x\in X$ and a neighborhood $O$ of $x$, there is a neighborhood
$U$ of $x$ such that $\overline{U}\subseteq O$. A space $X$ is \emph{regular} if it is a
$T_3$-space satisfying the $T_1$-separation axiom.

Let $X$ be a space with topology $\tau$. Then the family $\{\text{Int}_\tau\hskip1pt \overline{U}:
\emptyset\neq U\in \tau\}$ constitutes a base for a coarser topology $\sigma$ on $X$.
The space $X_{sr} = (X, \sigma)$ is called the \emph{semiregularization} of $X$. One
of the main results regarding the semiregularization of paratopological groups is the
following important theorem proved by Ravsky in \cite{Rav01}:

\begin{theorem}\label{Th1}
Let $G$ be an arbitrary paratopological group. Then the space $G_{sr}$
carrying the same group structure is a $T_3$ paratopological group. If
$G$ is Hausdorff, then $G_{sr}$ is a regular paratopological group.
\end{theorem}

The \emph{index of narrowness} of a paratopological group $G$ is denoted by $ib(G)$
(see \cite[Section~5.2]{AT}). By definition, $ib(G)$ is the smallest cardinal $\tau\geq \omega$
such that for each open neighborhood $U$ of the identity in $G$, there is a subset $A$ of $G$
satisfying $AU=UA=G$ and $|A|\leq \tau$. If $ib(G)=\omega$, then $G$ is called \emph{$\omega$-narrow.}

For a paratopological group $G$ with topology $\tau$, one defines the \textit{conjugate topology}
$\tau^{-1}$ on $G$ by $\tau^{-1}=\{U^{-1}: U\in \tau\}$. Then  $G' = (G, \tau^{-1})$ is
also a paratopological group, and the inversion $x\rightarrow x^{-1}$ is a homeomorphism
of $G$ onto $G'$. The upper bound $\tau^\ast=\tau\vee \tau^{-1}$ is a topological group topology
on $G$, and we say that the topological group $G^\ast=(G,\tau^\ast)$ is \textit{associated} to $G$.
A paratopological group $G$ is \emph{totally $\omega$-narrow} if the topological group $G^\ast$
associated to $G$ is $\omega$-narrow.

A paratopological group $G$ is called \emph{$\omega$-balanced} if for every neighborhood $U$ of the
identity $e$ in $G$ there is a countable family $\{V_n: n\in\omega\}$ of open neighborhoods of $e$ such
that for each $g\in G$, some $V_n$ satisfies $gV_n g^{-1}\subseteq U.$  In this case we say that the
family $\{V_n: n\in\omega\}$ is \emph{subordinated to $U$.} It is well known that every $\omega$-narrow topological group is $\omega$-balanced \cite[Proposition~3.4.10]{AT} and every totally $\omega$-narrow paratopological group is $\omega$-balanced \cite[Proposition~3.8]{ST2}.

In this paper, $c(X)$ and $\psi(X)$ stand for the cellularity and pseudocharacter
of $X$, respectively. The closure of a subset $Y$ of $X$ is denoted by $\overline{Y}$ or $\mathrm{cl}_X Y$
if we want to stress that the closure is taken in $X$. The cardinality of the continuum is
$\cont=2^\omega$.

\section{Characterizations of simply $sm$-factorizable (para)topological groups}\label{Sec:2}
In this section, we present some characterizations of simply $sm$-factorizable topological
and paratopological groups via continuous real-valued functions. The following notion plays
an important role in this article.

\begin{definition}
A (para)topological group $G$ is \emph{strongly submetrizable} if $G$ admits a coarser
separable metrizable (para)topological group topology or, equivalently, there exists a continuous
one-to-one homomorphism of $G$ onto a separable metrizable (para)topo\-logical group.
\end{definition}

It is clear that every strongly submetrizable (para)topological group is Hausdorff and
has countable pseudocharacter. Furthermore, the identity of a strongly submetrizable
paratopological group is the intersection of countably many \emph{closed} neighborhoods.

The following fact is obvious.

\begin{proposition}\label{le1}
\emph{Every  strongly submetrizable (para)topological group is simply $sm$-factorizable.}
\end{proposition}

Every $\omega$-narrow Hausdorff  topological group of countable pseudocharacter admits a
continuous one-to-one homomorphism onto a separable metrizable topological group (see \cite[Corollary~3.4.25]{AT}). Similarly, a totally $\omega$-narrow regular paratopological group
of countable pseudocharacter admits a continuous one-to-one homomorphism onto a separable
metrizable paratopological group (see \cite[Lemma~1.6]{PZ}). Thus we have:

\begin{lemma}\label{l1}
Every $\omega$-narrow Hausdorff topological group (totally $\omega$-narrow regular paratopological
group) of countable pseudocharacter is strongly submetrizable.
\end{lemma}

Now we give a characterization of simply $sm$-factorizable (para)topological groups in terms of
continuous homomorphisms to strongly submetrizable (para)topological groups as follows:

\begin{theorem}\label{th1}
Let $G$ be a (para)topological group. Then the following statements are equivalent:
\begin{enumerate}
\item[(1)] $G$ is simply $sm$-factorizable;
\item[(2)] for each continuous function $f\colon G\to \mathbb{R}$, one can find a continuous
homomorphism $\pi$ of $G$ onto a strongly submetrizable (para)topological group $H$ and a
continuous function $g\colon H\to \mathbb{R}$ such that $f=g\circ \pi$.
\item[(3)] for each continuous function $f\colon G\to \mathbb{R}$, one can find a continuous
homomorphism $\pi$ of $G$ onto a regular strongly submetrizable (para)topological group
$H$ and a continuous function $g\colon H\to \mathbb{R}$ such that $f=g\circ \pi$.
\end{enumerate}
\end{theorem}

\begin{proof}
(1)\,$\Rightarrow$\,(2). Let $\mathcal {V}$ be a countable base of $\mathbb{R}$ consisting co-zero sets
and $f$ be a continuous real-valued function on $G$. For every $V\in\mathcal{V}$, let $U_V=f^{-1}(V)$.
Then each element of the family $\mathcal {U}=\{U_{V}: V\in \mathcal {V}\}$ is a co-zero set in $G$. Since
$G$ is simply $sm$-factorizable, for each $V\in\mathcal{V}$ we can find a separable metrizable (para)topological group $H_{V}$ and a continuous homomorphism $\pi_V$ of $G$ onto $H_V$ such that $U_V=\pi_{V}^{-1}(\pi_V(U_V))$. Let $\pi=\Delta_{V\in\mathcal {V}} \pi_V$ be the diagonal mapping of
the family $\{\pi_V: V\in\mathcal{V}\}$ and $\Pi=\prod_{V\in\mathcal {V}} H_V$ be the topological product
of the family $\{H_V: V\in\mathcal{V}\}$. Clearly, $\pi\colon G\to h(G)\subseteq \Pi$ is a continuous homomorphism and $H'= \pi(G)$ is a separable metrizable (para)topological group. Now let $H$ have
the same group structure as $H'$ and endow $H$ with the quotient topology with respect to $h$. Then
clearly $H$ is strongly submetrizable and $\pi\colon G\to H$ is open. Thus it suffices to show that there
is a continuous function $g \colon H\to\mathbb{R}$ satisfying $f=g\circ \pi$. Since $\pi$ is continuous
and open, the latter will follow if we show that the equality $f(g_1)=f(g_2)$ holds for all $g_1,g_2\in G$
with $\pi(g_1)=\pi(g_2)$. Indeed, if $f(g_1)\neq f(g_2)$, then there is a $V\in \mathcal {V}$ such that
$f(g_2)\in V$ and $f(g_1)\notin V$.
Thus $g_2\in U_V$ and $g_1\notin U_V$. Since $U_{V}=\pi_V^{-1}(\pi_V(U_V))$, we have
$$
\pi^{-1}(\pi(g_2))=\bigcap_{W\in \mathcal {V}} \pi_W^{-1}(\pi_W(g_2))\subseteq
\pi_V^{-1}(\pi_V(U_V))=U_V.
$$
This implies that $\pi(g_1)\neq \pi(g_2)$ and completes the proof of the implication.\smallskip

(2)\,$\Rightarrow$\,(1). Let $U$ be a co-zero set in $G$. Then there is a continuous function
$f\colon G\to\mathbb{R}$ such that $U=f^{-1}(\mathbb{R}\setminus \{0\})$. By (2), one can find a
strongly submetrizable (para)topo\-logical group $H$, a continuous homomorphism $\pi$ of $G$
onto $H$ and a continuous function $g\colon H\to \mathbb{R}$ such that $f=g\circ \pi$. Let
$i\colon H\to H'$ be a continuous isomorphism onto a separable metrizable (para)topological
group $H'$. Then $\varphi=i\circ \pi$ is a continuous homomorphism of $G$ onto the separable
metrizable group $H'$ and $U=\varphi^{-1}(\varphi(U))$. This shows that $G$ is simply
$sm$-factorizable.

Now we show that $(2)\Leftrightarrow (3)$. Since every $T_0$-topological group is regular,
the equivalence $(2)\Leftrightarrow (3)$ for topological groups is obvious.

For paratopological groups, it suffices to show that $(2)\Rightarrow (3)$. Let $G$ be a paratopological
group  satisfying $(2)$. Let also $f$ be a continuous real-valued function on $G$. By (2), one can
find a continuous homomorphism $\pi\colon G\to H$ onto a strongly submetrizable paratopological group
$H$ and a continuous function $g$ on $H$ such that $f=g\circ\pi$. Then $H$ is a Hausdorff paratopological
group. Let $H_{sr}$ be the semiregularization of $H$. Then $H_{sr}$ is a regular paratopological group,
by Theorem~\ref{Th1}. From the fact that let $f\colon X\to Y$ be a continuous mapping of $X$ to a
regular space $Y$; then $f$ remains continuous as a mapping of the semiregularization $X_{sr}$ of
$X$ to $Y$ \cite[Lemma~3.5]{XST}, one can easily see that $H_{sr}$ is a strongly submetrizable paratopological group and $g\colon H_{sr}\to \mathbb{R}$ is continuous. Clearly, $f=g\circ (i\circ \pi)$,
where $i\colon H\to H_{sr}$ is the identity mapping. The proof is complete.
\end{proof}

A space $X$ is called \emph{weakly Lindel\"{o}f} if for each open cover $\mathcal {U}$ of $X$, there
exists a countable subfamily $\mathcal {V}$ of $\mathcal {U}$ such that $\bigcup\mathcal {V}$ is
dense in $X$.

\begin{corollary}\cite[Proposition 5.18]{AT1}\label{C1}
Every weakly Lindel\"{o}f topological group $G$ is simply $sm$-factorizable.
\end{corollary}

\begin{proof}
Take any continuous function $f\colon G\to \mathbb{R}$.  According to \cite[Theorem~8.1.18]{AT},
one can find a continuous homomorphism $\pi\colon G \to H$ onto a topological group $H$ such that
the pseudocharacter of $H$ is countable and a continuous real-valued function $h$ on $H$ such that
$f=h\circ \pi$. Every weakly Lindel\"{o}f topological group is $\omega$-narrow \cite[Proposition~5.2.8]{AT}.
Hence the group $G$ and its continuous homomorphic image $H$ are $\omega$-narrow as well.
According to Lemma~\ref{l1} $H$ is a strongly submetrizable topological group and therefore, $G$
is simply $sm$-factorizable by Theorem~\ref{th1}.
\end{proof}

Let $\varphi\colon G\to H$ be a continuous surjective homomorphism of semitopological groups.
The pair $(H, \varphi)$ is called a \emph{$T_2$-reflection} of $G$ if $H$ is a Hausdorff
semitopological group and for every continuous mapping $f\colon G \to X$ of $G$ to a Hausdorff space
$X$, there exists a continuous mapping $h\colon H\to X$ such that $f = h\circ\varphi$. Abusing of
terminology we say that $T_2(G)$ is the $T_2$-reflection of $G$, thus omitting the corresponding homomorphism $\varphi$ (see \cite{T2}). The homomorphism $\varphi$ is denoted by $\varphi_{ G, 2}$
and called the \emph{canonical homomorphism} of $G$ onto $T_2(G)$.

\begin{proposition}\label{Po1}
\emph{A paratopological group $G$ is simply $sm$-factorizable if and only if so is the  $T_2$-reflection
$T_2(G)$ of $G$.}
\end{proposition}

\begin{proof}
Let $G$ be simply $sm$-factorizable. Take any continuous function $f\colon T_2(G)\to \mathbb{R}$.
Then $f\circ \varphi_{G,2}$ is continuous real-valued function on $G$, and therefore, we can find a strongly submetrizable paratopological group $H$, a continuous homomorphism $\pi$ of $G$ onto $H$ and a
continuous function $g\colon H\to \mathbb{R}$ such that $f\circ \varphi_{G,2}=g\circ\pi$, by Theorem~\ref{th1}. Since $H$ is a Hausdorff paratopological group, there is continuous map $p\colon T_2(G)\to H$ such that $p\circ \varphi_{G,2}=\pi$. Observing that $\varphi_{G,2}$ and $\pi$ are homomorphisms, one can easily
show that $p$ is also a homomorphism. Clearly, the subgroup $p(T_2(G))$ of $H$ is strongly submetrizable and $f=g\circ{p}$, so $T_2(G)$ is simply $sm$-factorizable by Theorem~\ref{th1}.

Let $T_2(G)$ be simply $sm$-factorizable. Take any continuous function $f\colon G\to \mathbb{R}$.
Since $\mathbb{R}$ is a Hausdorff space, there is continuous function $h\colon T_2(G)\to \mathbb{R}$
such that $h\circ \varphi_{G,2}=f$. Further, combining Theorem~\ref{th1} and the fact that $T_2(G)$ is simply $sm$-factorizable we see that there are a strongly submetrizable paratopological group $H$, a continuous homomorphism $\pi$ of $T_2(G)$ onto $H$ and a continuous function $g\colon H\to \mathbb{R}$ such that $h=g\circ\pi$. Thus we have the equality
$$
f=h\circ \varphi_{G,2}=g\circ (\pi\circ \varphi_{G,2}).
$$
 By Theorem~\ref{th1}, this implies that $G$ is simply $sm$-factorizable.
\end{proof}

\begin{proposition}\label{Po2}
\emph{A paratopological group $G$ is simply $sm$-factorizable if and only if so is the semiregularization
$G_{sr}$ of $G$.}
\end{proposition}

\begin{proof}
Let $G$ be simply $sm$-factorizable. Take any continuous function $f\colon G_{sr}\to \mathbb{R}$.
Clearly, $f$ is continuous on $G$, and applying Theorem~\ref{th1} we find a regular strongly submetrizable paratopological group $H$, a continuous homomorphism $\pi$ of $G$ onto $H$ and a continuous function $g\colon H\to \mathbb{R}$ such that $f\circ \varphi_{G,2}=g\circ\pi$. Since $H$ is regular, by \cite[Lemma~3.5]{XST} $\pi$ is also continuous on $G_{sr}$. Hence $G_{sr}$ is simply $sm$-factorizable according to Theorem~\ref{th1}.

Let $G_{sr}$ be simply $sm$-factorizable. Take any continuous function $f\colon G\to \mathbb{R}$.
Then by \cite[Lemma~3.5]{XST}, $f$ is also continuous on $G_{sr}$. So Theorem~\ref{th1} implies
that there are strongly submetrizable paratopological group $H$, a continuous homomorphism $\pi$
of $G_{sr}$ onto $H$ and a continuous function $g\colon H\to \mathbb{R}$ such that $f\circ \varphi_{G,2}=g\circ\pi$. Thus we have that $f=g\circ (\pi\circ i)$, where $i\colon G\to G_{sr}$ is the identity
mapping. By Theorem~\ref{th1}, this implies that $G$ is simply $sm$-factorizable.
\end{proof}

Let $f\colon X \to Y$ be a continuous mapping. Then $f$ is said to be \emph{$d$-open}
if for each open subset $O$ of $X$ there exists an open subset $V$ of $Y$ such that $f(O)$ is a
dense subset of $V$ or, equivalently, $f(O)$ is a subset of the interior of the closure of $f(O)$ in $Y$.

We recall that a Hausdorff paratopological group $G$ has countable \emph{Hausdorff number}, in symbols
$Hs(G)\leq\omega$, if for each neighborhood $U$ of the identity $e$ in $G$ there is a countable family
$\{U_n: n\in\omega\}$ of open neighborhoods of $e$ such that $\bigcap_{n\in \omega} U_nU_n^{-1}
\subseteq U$.

\begin{lemma}\label{Le1}
Let $G$ be a Hausdorff weakly Lindel\"{o}f paratopological group with $Hs(G)\leq \omega$. If $G$ is
$\omega$-balanced, then for any continuous real-valued function $f$ on $G$ one can find a $d$-open homomorphism $\pi\colon G\to K$ onto a regular paratopological group $K$ of countable pseudocharacter
and a continuous function $h\colon K\to \mathbb{R}$ such that $f=h\circ \pi$.
\end{lemma}

\begin{proof}
Let $f\colon G\to\mathbb{R}$ be a continuous function. By \cite[Theorem~2.4]{PZ}, we can find an open
continuous homomorphism $\pi\colon G\to K$ of $ G$ onto a Hausdorff paratopological group $K$,
a continuous function $h$ on $K$ and a sequence $\{V_n: n \in \omega \}$ of open neighborhoods of the
identity $e_K$ in $K$ such that $f=h\circ\pi$ and $\{e_K\}=\bigcap_{n\in \omega}\overline{V_n}$. Let
$K_{sr}$ be the semiregularization of $K$. Since $K$ is a Hausdorff paratopological group, $K_{sr}$
is a regular paratopological group by Theorem~\ref{Th1}. Clearly, the elements of the sequence
$\{\mathrm{Int}_K \mathrm{cl}_K{V_n}: n\in \omega \}$ are open neighborhoods of the identity
$e_K$ in $K_{sr}$ and $\{e_K\}=\bigcap_{n\in \omega}\mathrm{Int}_K\mathrm{cl}_K V_n$. This
implies that $K_{sr}$ has countable pseudocharacter. By \cite[Lemma~3.5]{XST}, $h$ remains
continuous on $K_{sr}$. Since the identity mapping $i\colon X\to X_{sr}$ is $d$-open and continuous
for any space $X$ \cite[Lemma~3.2]{XY}, one can easily check that $i\circ\pi\colon G\to K_{sr}$ is
a $d$-open continuous homomorphism. Clearly, $f=h\circ (i\circ\pi)$. This completes the proof.
\end{proof}

\begin{corollary}[See Theorem~2.9 in \cite{PZ}]\label{C2.8}
Every totally $\omega$-narrow weakly Lindel\"{o}f paratopological group $G$ is simply
$sm$-factorizable.
\end{corollary}

\begin{proof}
Clearly weak Lindel\"{o}fness and total $\omega$-narrowness are preserved by continuous maps and continuous homomorphisms, respectively. Theorem~\ref{Th1} implies that $(T_2(G))_{sr}$ is a regular
weakly Lindel\"{o}f and totally $\omega$-narrow paratopological group. According to Propositions~\ref{Po1} and~\ref{Po2}, $G$ is simply $sm$-factorizable if and if so is $(T_2(G))_{sr}$. Hence we can assume that $G$ is regular.

Take any continuous function $f\colon G\to \mathbb{R}$. Since every regular totally $\omega$-narrow paratopological group $H$ is $\omega$-balanced (see \cite[Proposition~3.8]{ST2}) and satisfies $Hs(H)\leq \omega$ (\cite[Theorem~2]{Sa}), we apply Lemma~\ref{Le1} to find a $d$-open continuous homomorphism
$\pi\colon G\rightarrow K$ onto a regular paratopological group $K$ with countable pseudocharacter and
a continuous function $h\colon K\to \mathbb{R}$ such that $f=h\circ \pi$. Clearly, $K$ is totally
$\omega$-narrow, so $K$ is strongly submetrizable by Lemma~\ref{l1}. Therefore, $G$ is simply
$sm$-factorizable by Theorem~\ref{th1}.
\end{proof}

It is well known that a subgroup $H$ of an $\mathbb{R}$-factorizable topological group $G$ is
$z$-embedded in $G$ if and only if $H$ is $\mathbb{R}$-factorizable. Now we consider
the $z$-embedded subgroups of simply $sm$-factorizable (para)topological groups.

\begin{proposition}\label{P}
\emph{A (para)topological group $G$ is simply $sm$-factorizable if and only if for each continuous function
$f\colon G\rightarrow \mathbb{R}^{\omega}$, there exist a continuous homomorphism $\pi\colon G\to H$
onto a (regular) strongly submetrizable (para)topological group $H$ and a continuous function $g\colon
H\to \mathbb{R}^{\omega}$ such that $f=g\circ\pi$.}
\end{proposition}

\begin{proof}
According to Theorem~\ref{th1} it suffices to prove the necessity. Take any continuous function
$f\colon G\to \mathbb{R}^{\omega}$. For every $i\in\omega$, denote by $p_i$ the projection of $\mathbb{R}^{\omega}$ to the $i$th factor. Then $p_i\circ f\colon G\to \mathbb{R}$ is a continuous real-valued function. Since $G$ is simply $sm$-factorizable, there are a (regular) strongly submetrizable (para)topological group $H_i$, a continuous homomorphism $\pi_i$ of $G$ onto $H_i$ and a continuous function $g_i\colon H_i\to \mathbb{R}$ such that $p_i \circ f=g_i\circ \pi_i$, for each $i\in\omega$. Denote by $\pi$ the diagonal product
of the family $\{\pi_i: i\in\omega\}$. Then $\pi\colon G\to \prod_{i\in\omega}H_i$ is a continuous homomorphism and the image $K=\pi(G)$ is a (regular) strongly submetrizable (para)topological group. For each $i\in\omega$, let  $q_i\colon \prod_{j\in\omega}H_j\to H_i$ be the projection. Then $\pi_i=q_i\circ\pi$. Finally, denote by $g^\ast$ the Cartesian product of the family $\{g_i: i\in\omega\}$. Then $g^\ast\colon\prod_{i\in\omega}H_i\to \mathbb{R}^\omega$ is continuous. Let us verify that $f=g^\ast\circ \pi$. Indeed, for each
$i\in\omega$ and each $x\in G$, we have:
$$
p_i\circ f=g_i\circ \pi_i=g_i\circ q_i\circ \pi=p_i\circ g^\ast\circ \pi.
$$
Hence the function $g=g^\ast \res_K$ satisfies $f=g\circ \pi$.
\end{proof}

\begin{theorem}\label{Th:sm-f}
Let $G$ be a simply $sm$-factorizable (para)topological group. If a subgroup $H$ of $G$ is $z$-embedded
in $G$, then $H$ is simply $sm$-factorizable.
\end{theorem}

\begin{proof}
Consider a continuous function $f\colon H\to \mathbb{R}$. Let $\mathcal {V}$ be a countable base
of $\mathbb{R}$ consisting of co-zero sets and $U_V=f^{-1}(V)$, where $V\in\mathcal{V}$. Then
each element of the family $\mathcal {U}=\{U_V: V\in\mathcal {V}\}$ is co-zero set in $H$. Since
$H$ $z$-embedded in $G$, for each $V\in \mathcal {V}$ there is a continuous function $g_{V}\colon
G\to \mathbb{R}$ such that $g_{V}^{-1}(\mathbb{R}\setminus \{0\})\cap H=U_V$. Denote
by $g$ the diagonal product of the family $\{g_{V}: V\in \mathcal{V}\}$. Then $g\colon G\to
\mathbb{R}^{\mathcal{V}}$ is continuous. Since $\mathcal{V}$ is countable and $G$ is simply
$sm$-factorizable, it follows from Proposition~\ref{P} that one can find a continuous homomorphism
$\varphi$ of $G$ onto a strongly submetrizable (para)topological group $K$ and a continuous function
$h\colon K\to \mathbb{R}^{\mathcal{V}}$ such that $g=h\circ\varphi$. Let us verify the following:\medskip

{\bf Claim.} \emph{If $x,y\in G$ and $\varphi(x)=\varphi(y)$, then $f(x)=f(y)$.}\medskip

If $f(x)\neq f(y)$, then there is $V\in\mathcal {V} $ such that $f(x)\in V$ and $f(y)\notin V$.
Hence our choice of the function $g_V$ implies that $g_V(x)\neq 0$ and $g_V(y)=0$.
Therefore, $g(x)\neq g(y)$ and, hence, $\varphi(x)\neq \varphi(y)$ since $g=h\circ \varphi$.
This proves the claim.

According to the above Claim, there is a function $j\colon \varphi(H)\to \mathbb{R}$ such that
$f=j\circ \varphi \res_H$ ($j$ can fail to be continuous when $\varphi(H)$ is endowed with
the topology inherited from $K$). Denote by $L$ the group $\varphi(H)$ endowed with the quotient
topology with respect to the homomorphism $\varphi$. Then $\varphi \res_H\colon H\to L$ is open,
so the function $j\colon L\to \mathbb{R}$ is continuous. Clearly, the group $L$ is strongly
submetrizable. This completes the proof of the theorem.
\end{proof}

Theorem~\ref{Th:sm-f} makes it natural to ask the following question. Let $H$ be a subgroup
of a topological group $G$. Is $H$ $z$-embedded in $G$ provided both $H$ and $G$ are simply
$sm$-factorizable? It turns out that the answer to the question is \lq\lq{No\rq\rq}.

\begin{example}
\emph{There exists a separable (hence $\omega$-narrow) simply $sm$-factorizable topological
group $G$ which contains a simply $sm$-factorizable subgroup $H$ such that it fails to be
$z$-embedded in $G$.}
\end{example}

\begin{proof}
Let $H$ be a separable topological group which fails to be $\mathbb{R}$-factorizable. One can take
as $H$ the free Abelian topological group over the Sorgenfrey line \cite{RS}. Then $H$ has countable
cellularity, so  \cite[Corollary~5.19]{AT1} implies that $H$ is simply $sm$-factorizable. Every separable
topological group is $\omega$-narrow, so $H$ embeds as a topological subgroup into a product
$G=\prod_{\alpha\in A} G_\alpha$ of second-countable topological groups \cite[Theorem~3.4.23]{AT}.
Since the weight of a separable topological group is at most $\cont=2^\omega$, we can assume
without loss of generality that $|A|\leq\cont$. Then the product group $G$ is separable, while
\cite[Corollary~8.1.15]{AT} implies that $G$ is $\mathbb{R}$-factorizable. However, the subgroup
$H$ of $G$ cannot be $z$-embedded in $G$\,---\,otherwise $H$ would be $\mathbb{R}$-factorizable
by \cite[Theorem~8.2.6]{AT}.
\end{proof}

In Theorem~\ref{Th3} below we give a characterization of simply $sm$-factorizable topological groups
assuming that the groups are $\omega$-narrow. First we recall the notion of \emph{admissible}
subgroup introduced in \cite{Tk89} (see also \cite[Section~5.5]{AT}.

\begin{definition}\label{Def:1}
Let $G$ be a topological group and $\{U_n: n\in \omega\}$ a sequence of open symmetric
neighborhoods of the identity in $G$ such that $U_{n+1}^2\subseteq U_n$, for each
$n\in \omega$. Then $N=\bigcap_{n\in \omega}U_n$ is a subgroup of $G$ which is called
\emph{admissible.}
\end{definition}

It follows from the above definition that every admissible subgroup of a topological group
$G$ is closed and that every neighborhood of the identity in $G$ contains an admissible
subgroup \cite[Lemma~5.5.2]{AT}.

\begin{lemma}\label{L1}
Let $f\colon G\to H$ be a continuous homomorphism of topological groups. If $H$ has
countable pseudocharacter, then $\ker{f}$ is an invariant admissible subgroup of $G$.
\end{lemma}

\begin{proof}
Clearly, $N=\ker{f}$ is an invariant subgroup of $G$, so it suffices to show that $N$ is
admissible. Since $H$ is a topological group with countable pseudocharacter, one can find
a sequence $\{U_n: n\in \omega\}$ of open symmetric neighborhoods of the identity $e$ in $H$
such that $U_{n+1}^2\subseteq U_n$, for each $n\in \omega$ and $\{e\}=\bigcap_{n\in \omega}U_n$.
Let $V_n=f^{-1}(U_n)$, $n\in\omega$. Then $\{V_n: n\in \omega\}$ is a sequence of open symmetric neighborhoods of the identity in $G$ satisfying $V_{n+1}^2\subseteq V_n$, for each $n\in \omega$
and $N=\bigcap_{n\in \omega} V_n$. This implies that $N$ is an admissible subgroup of $G$.
\end{proof}

\begin{theorem}\label{Th3}
The implication {\rm (a)}\,$\Rightarrow$\,{\rm (b)} is valid for every topological group $G$, where
\begin{enumerate}
\item[{\rm (a)}] $G$ is simply $sm$-factorizable;
\item[{\rm (b)}] for every continuous function $f\colon G\to \mathbb{R}$, there exists an invariant
admissible subgroup $N$ of $G$ such that $f$ is constant on $gN$, for each $g\in G$.
\end{enumerate}
Furthermore, if $G$ is $\omega$-narrow, then {\rm (a)} and {\rm (b)} are equivalent.
\end{theorem}

\begin{proof}
Let us show that {\rm (a)}\,$\Rightarrow$\,{\rm (b)}. Assume that $G$ is a simply $sm$-factorizable
topological group. Then, according to Theorem~\ref{th1}, one can find a continuous homomorphism
$\pi$ of $G$ onto a strongly submetrizable topological group $H$ and a continuous function $g\colon H\to \mathbb{R}$ such that $f=g\circ \pi$. Clearly, $H$ has countable pseudocharacter, so by Lemma~\ref{L1},
the kernel $N=\pi^{-1}(e)$ of $\pi$ is an invariant admissible subgroup of $G$. Since $f=g\circ \pi$, one
can easily see that $f$ is constant on $xN$, for each $x\in G$. This implies (b).

Assume that $G$ is an $\omega$-narrow group satisfying (b), and let $f\colon G\to \mathbb{R}$ be a
continuous function. Then there is an invariant admissible subgroup $N$ of $G$ such that $f$ is constant
on $xN$ for each $x\in G$. Let $G/N$ be the quotient topological group of $G$ and $\pi\colon G\to G/N$ be
the quotient homomorphism. Since $G$ is $\omega$-narrow and so is every quotient group of $G$,
$G/N$ is an $\omega$-narrow group of countable pseudocharacter (see \cite[Lemma~2.3.6]{T1}).
By Lemma~\ref{l1}, $G/N$ is a strongly submetrizable topological group. Observing that $f$ is constant on $xN$ for each $x\in G$ and $\pi\colon G\to G/N$ is open, one can find a continuous function $h\colon
G/N\to \mathbb{R}$ such that $f=h\circ \pi$. From Theorem~\ref{th1} it follows that $G$ is simply $sm$-factorizable.
\end{proof}

We have just established the equivalence of items (a) and (b) in Theorem~\ref{Th3} under the additional
assumption of the $\omega$-narrowness of $G$. Since every discrete abelian group of cardinality
$2^\omega$ is simply $sm$-factorizable \cite[Proposition~5.15]{AT1}, we see that simply $sm$-factorizable topological groups need not be $\omega$-narrow. Therefore, it is natural to ask whether every simply $sm$-factorizable topological group is $\omega$-balanced. In the next example we answer this question in the negative.

\begin{example}
\emph{Let $H=GL(\mathbb{R},2)$ be the group of $2\times 2$ invertible matrices with
real entries. Denote by $G$ the group $H^\omega$ endowed with the box topology.
Then $G$ is a strongly submetrizable (hence simply $sm$-factorizable) group of
countable pseudocharacter which fails to be $\omega$-balanced.}
\end{example}

\begin{proof}
Indeed, denote by $G_\ast$ the group $H^\omega$ with the usual product topology. Then the identity
mapping of $G$ onto $G_\ast$ is a continuous isomorphism onto a separable metrizable topological
group, so $G$ is strongly submetrizable, hence simply $sm$-factorizable (Proposition~\ref{le1}). However,
 $G$ is not $\omega$-balanced. To show this we slightly modify the argument from \cite[Example~2]{Pes}.
 It is well known that the group $H$ contains two sequences $\{a_n: n\in\omega\}$ and $\{b_n: n\in\omega\}$ and an element $z_0\neq e$ such that $a_nb_n\to e$ and $b_na_n\to z_0$ for $n\to\infty$, where $e$ is
 the identity element of $H$ \cite[4.24]{HR}. Choose an open neighborhood $U_0$ of $e$ in $H$ such that $z_0\notin \overline{U_0}$. It is easy to see that for every open neighborhood $V$ of $e$ in $H$, there exists $k\in\omega$ such that $b_k V b_k^{-1}\setminus U_0\neq\emptyset$\,---\,otherwise $b_k a_k=b_k(a_k b_k) b_k^{-1}\in b_k V b_k^{-1}\subset U_0$ for all sufficiently big $k\in\omega$, which contradicts our
 choice of the sequences $\{a_n: n\in\omega\}$, $\{b_n: n\in\omega\}$ and of the set $U_0$.

Clearly $U=U_0^\omega$ is an open neighborhood of the identity in $G$. Suppose for a contradiction that
$G$ is $\omega$-balanced. Then there exists a sequence $\{V_n: n\in\omega\}$ of open neighborhoods
of the identity in $G$ subordinated to $U$. One can assume without loss of generality that every $V_n$
has the form $\prod_{k\in\omega} V_{n,k}$, where $V_{n,k}$ is an open neighborhood of $e$ in $H$ for
each $k\in\omega$. We have just shown that for every $n\in\omega$, there exists $k_n\in\omega$ such
that $b_{k_n}V_{n,n}b_{k_n}^{-1}\setminus U_0\neq\emptyset$. Let $x=(x_n)_{n\in\omega}$ be the
element of $G$ defined by $x_n=b_{k_n}$ for each $n\in\omega$. Then $xV_nx^{-1}\setminus U\neq\emptyset$ for each $n\in\omega$ since the projections of the sets $xV_nx^{-1}$ and $U$ to
the $n$th factor are $b_{k_n}V_{n,n}b_{k_n}^{-1}$ and $U_0$, respectively. This contradicts our
choice of the sequence $\{V_n: n\in\omega\}$. Therefore, the group $G$ is not $\omega$-balanced.
\end{proof}

The following result partially answers Problem~\ref{Q1.2}.

\begin{corollary}\label{Cor:Adm}
Let $\pi\colon G\to H$ be a continuous homomorphism of topological groups such that for
each invariant admissible subgroup $K\subseteq G$, the image $\pi(K)$ contains an invariant
admissible subgroup of $H$. If $G$ is simply $sm$-factorizable and $H$ is $\omega$-narrow,
then $H$ is simply $sm$-factorizable.
\end{corollary}

\begin{proof}
Let $f\colon H\to \mathbb{R}$ be any continuous function. Since $G$ is simply $sm$-factorizable, it
follows from Theorem~\ref{Th3} that there is an invariant admissible subgroup $N$ of $G$ such that
for each $x\in G$, $f\circ\pi$ is constant on $xN$. According to our assumption $\pi(N)$ contains an
invariant admissible subgroup $K$ of $H$. Since $\pi$ is a homomorphism and for each $x\in G$,
$f\circ\pi$ is constant on $xN$, one can easily verify that for each $y\in H$, $f$ is constant on $yK$.
Observing that $H$ is $\omega$-narrow, from Theorem~\ref{Th3} it follows that $H$ is simply
$sm$-factorizable.
\end{proof}

\begin{remark}
The condition on the homomorphism $\pi$ in Corollary~\ref{Cor:Adm} is quite strong. It can easily
fail, even if the homomorphism $\pi$ is open. Indeed, according to \cite[Theorem~7.6.18]{AT}, every
Hausdorff topological group $H$ is a quotient group of a topological group $G$ with countable
pseudocharacter. So we can take $H$ to be a Hausdorff topological group with $\psi(H)>\omega$
and find an open continuous surjective homomorphism $\pi\colon G\to H$, where $G$ is a topological
group of countable pseudocharacter. Then $K=\{e\}$ is an invariant admissible subgroup of $G$,
where $e$ is the identity element of $G$. Clearly $\pi(K)=\{e_H\}$ does not contain any admissible
subgroup of $H$. It also follows from Proposition~\ref{Prop:New} below that $G$ cannot be simply
$sm$-factorizable if $|H|>\cont$.
\end{remark}

As we mentioned after Definition~\ref{Def:1}, every neighborhood of the identity in a topological
group contains an admissible subgroup \cite[Lemma~5.5.2]{AT}. This conclusion can be strengthened
for $\omega$-balanced topological groups:

\begin{lemma}\label{L2.17}
Every neighborhood of the identity $e$ in an $\omega$-balanced topological group $G$ contains
an invariant admissible subgroup.
\end{lemma}

\begin{proof}
Every $\omega$-balanced topological group is a subgroup of a topological product of first countable
topological groups \cite[Theorem~5.1.9]{AT}, so for each open neighborhood $U$ of $e$ in $G$, one
can find a continuous homomorphism $p$ on $G$ onto a first countable topological group $H$ and
an open neighborhood $V$ of the identity $e_H$ in $H$ satisfying $p^{-1}(V)\subseteq U$. Clearly $K=\overline{\{e_H\}}$ is an invariant admissible subgroup of $H$ satisfying $K\subset V$, so
$p^{-1}(K)$ is an invariant admissible subgroup of $G$ contained in $U$.
\end{proof}

We recall that $X$ is a \emph{$P$-space} if every $G_\delta$-set in $X$ is open.  Similarly,
a (para)topo\-log\-i\-cal group $G$ is said to be a \emph{$P$-group} if the underlying space of $G$
is a $P$-space. The following result gives a partial answer to Problem~\ref{Q1.3}.

\begin{corollary}
If an $\omega$-narrow topological group $H$ is a quotient group of a simply $sm$-factorizable
$P$-group, then $H$ is $\mathbb{R}$-factorizable.
\end{corollary}

\begin{proof}
Let $\pi\colon G\to H$ be a quotient homomorphism and $G$ be a simply $sm$-factorizable $P$-group.
Then $H$ is a $P$-group by \cite[Lemma~4.4.1\,c)]{AT}. Since every $\omega$-narrow simply
$sm$-factorizable $P$-group is $\mathbb{R}$-factorizable \cite[Proposition~5.23]{AT1}, it suffices
to show that $H$ is simply $sm$-factorizable. Let $f$ be a continuous real-valued function on $H$.
Since $G$ is simply $sm$-factorizable, Theorem~\ref{Th3} implies that there is an invariant admissible
subgroup $N$ of $G$ such that for each $x\in G$, $f\circ\pi$ is constant on $xN$. Therefore, $f$ is
constant on $y\pi(N)$ for each $y\in H$. Observing that $G$ is a $P$-space and $\pi$ is open, we
see that $\pi(N)$ is an open neighborhood of the identity in $H$. Clearly $H$ is $\omega$-balanced,
so Lemma~\ref{L2.17} implies that $\pi(N)$ contains an invariant admissible subgroup $K$ of $H$.
It is clear that $f\circ\pi$ is constant on $yK$ for each $y\in H$, so $H$ is simply $sm$-factorizable
by Theorem~\ref{Th3}, because $H$ is $\omega$-narrow.
\end{proof}

\begin{proposition}\label{Prop:New}
\emph{Every regular simply $sm$-factorizable (para)topological group $G$ of countable
pseudocharacter admits a continuous isomorphic bijection onto a separable metrizable
(para)topological group. Therefore, $G$ is strongly submetrizable and $|G|\leq \cont$.}
\end{proposition}

\begin{proof}
Every regular paratopological group is completely regular \cite{BR}.
Let $\{U_n: n\in \omega\}$ be a family open neighborhoods of the identity $e$ in $G$ such that
$\{e\}=\bigcap_{n\in \omega}U_n$. For every $n\in \omega$, one can find a continuous function
$f_n\colon G\to \mathbb{R}$ such that $f_n(e)=0$ and $f_n(x)=1$ for each $x\in G\setminus U_n.$
Denote by $f$ the diagonal product of the family $\{f_n: n\in \omega\}$. Then $f\colon G \to
\mathbb{R}^\omega$ is continuous. Since $G$ is simply $sm$-factorizable, Proposition~\ref{P}
implies that we can find a continuous homomorphism $p\colon G\to H$ onto a strongly submetrizable (para)topological group $H$ and a continuous map $h\colon H\to\mathbb{R}^\omega$ such that
$f=h\circ p$. Note that $h(e_H)=f(e)=\bar{0}=(0,0,\ldots)$.

We claim that $p$ is a continuous isomorphic bijection. Indeed, it suffices to show that
$\ker{p}= \{e\}$. Take an element $x\in G\setminus \{e\}$. Then there is $n\in \omega$
such that $x\notin U_n$, so $f_n(x)=1$. Therefore, $f(x)\neq \bar{0}$. Hence $x\notin \ker{p}$
because $f=h\circ p$ and $h(e_H) = \bar{0}$.

Since every separable metrizable space has cardinality at most $\cont$ and $p$ is a
bijection, we see that $|G|\leq\cont$.
\end{proof}

\section{Completions of simply $sm$-factorizable topological groups}\label{Sec:3}
In this section we study the Dieudonn\'e and Hewitt--Nachbin completions of simply
$sm$-factorizable topological groups. In Theorem~\ref{Th3.2} we answer Problem~\ref{Q1.1}
affirmatively. In fact, we prove a stronger result: Every Hausdorff simply $sm$-factorizable
topological group is \emph{completion friendly.}

A subset $Y$ of a space $X$ is \emph{$G_\delta$-dense} in $X$ if every nonempty
$G_\delta$-set in $X$ intersects $Y$. The biggest set  $Z\subset X$ which contains
$Y$ as a $G_\delta$-dense subset is called the \emph{$G_\delta$-closure} of $Y$
in $X$. A space $X$ is called \emph{Moscow} if for each open set $U$ of $X$, the
closure $\overline{U}$ of $U$ is the union of a family of $G_\delta$-sets in $X$.

Let $\varrho{G}$ be the Ra\u{\i}kov completion of a topological group $G$. We denote the
$G_\delta$-closure of $G$ in $\varrho{G}$ by $\varrho_\omega{G}$. It is easy to verify that
$\varrho_\omega{G}$ is a subgroup of $\varrho{G}$ (see \cite[Section~6.4]{AT}).

\begin{proposition}\label{P3.1}
\emph{Let $H$ be a $G_\delta$-dense simply $sm$-factorizable subgroup of a topological
group $G$. Then $H$ is $C$-embedded in $G$.}
\end{proposition}

\begin{proof}
Take any continuous function $f\colon H\to \mathbb{R}$. According to Theorem~\ref{th1},
we can find a continuous homomorphism $\pi$ of $H$ onto a strongly submetrizable topological
group $F$ and a continuous function $g\colon F\to \mathbb{R}$ such that $f=g\circ \pi$. Let
$\varrho{G}$ and $\varrho{F}$ be the Ra\u{\i}kov completions of $G$ and $F$, respectively. Since
$H$ is dense in $\varrho{G}$, $\pi$ extends to a continuous homomorphism $\tilde{\pi}\colon
\varrho{G}\to \varrho{F}$. Denote by $\varrho_\omega{G}$ and $\varrho_\omega{F}$ the
$G_\delta$-closures of $G$ and $F$ in $\varrho{G}$ and $\varrho{F}$, respectively. Since $F$
is strongly submetrizable, it has countable pseudocharacter and, hence, $F$ is a Moscow space \cite[Corollary~6.4.11(1)]{AT}. It is well known that if a Moscow space $Y$ is a $G_\delta$-dense
subspace of a homogeneous space $X$, then $X$ is also a Moscow space and $Y$ is $C$-embedded
in $X$ \cite[Theorem~6.1.8]{AT}. Since $\varrho_\omega{F}$ is a topological group and $F$ is a Moscow
space, $F$ is $C$-embedded in $\varrho_\omega{F}$. Hence $g$ extends to a continuous function
$\tilde{g}\colon \varrho_\omega{F}\to \mathbb{R}$. Since $H$ is $G_\delta$-dense in $G$, we have the
inclusion $\tilde{\pi}(G)\subseteq \varrho_\omega{F}$. Thus $\tilde{g}\circ \tilde{\pi} \res_{G}$
is a continuous extension of $f$. This proves that $H$ is $C$-embedded in $G$.
\end{proof}

A topological group $G$ is called \emph{completion friendly} if $G$ is $C$-embedded in
$\varrho_\omega{G}$. Every completion friendly group is a $PT$-group \cite[Proposition~6.5.17]{AT}.

The following result gives a positive answer to Problem~\ref{Q1.2}.

\begin{theorem}\label{Th3.2}
Let $G$ be a Hausdorff simply $sm$-factorizable topological group. Then the equalities
$\mu{G}=\varrho_\omega{G}=\upsilon{G}$ hold and, therefore, $G$ is completion friendly.
\end{theorem}

\begin{proof}
Since every simply $sm$-factorizable topological group $H$ satisfies $ib(H)\leq\cont$
\cite[Proposition~5.14]{AT1}, we have that $ib(G)\leq \cont$. According to
\cite[Theorem~5.4.10]{AT}, the inequality $c(H)\leq 2^{ib(H)}$ holds for every topological group
$H$. We conclude, therefore, that $c(G)\leq 2^\cont$. From \cite[Theorem~6.2.2]{AT}
it follows that the cardinal number $2^\cont$ is Ulam non-measurable, so the cardinality
of every discrete family of open sets in $G$ is Ulam non-measurable and hence the equality
$\mu{G}=\upsilon{G}$ holds by \cite[Lemma~8.3.1]{AT}.

Note that the space $\varrho_\omega{G}$ is Dieudonn\'e complete \cite[Proposition~6.5.2]{AT}.
Since $\mu{G}=\upsilon{G}$ and $G$ is $C$-embedded in the Dieudonn\'{e} complete group
$\varrho_\omega{G}$ by Proposition~\ref{P3.1}, we conclude that $\mu{G}=\varrho_\omega{G}=
\upsilon{G}$.
\end{proof}

\begin{corollary}\label{C3.3}
Every Hausdorff simply $sm$-factorizable topological group $G$ is a $PT$-group, so the Dieudonn\'{e}
completion $\mu{G}$ of the space $G$ admits a natural structure of topological group containing $G$
as a dense subgroup.
\end{corollary}

By Corollary~\ref{C1} and Theorem~\ref{Th3.2}, we obtain the following result:

\begin{corollary}[See Proposition~2.4 of \cite{ST}]
Every Hausdorff weakly Lindel\"{o}f topological group $G$ satisfies the equalities
$\mu{G}=\varrho_\omega{G}=\upsilon{G}$ and, therefore, $G$ is completion friendly.
\end{corollary}

We also present an alternative proof of \cite[Theorem~5.21]{AT1}:

\begin{corollary}\label{Cor:3.5}
Let $G$ be a simply $sm$-factorizable topological group which is $C$-embedded in a regular
Lindel\"{o}f space $X$. Then the group $G$ is $\mathbb{R}$-factorizable and the closure of $G$
in $X$ is a topological group containing $G$ as a dense subgroup.
\end{corollary}

\begin{proof}
We can assume without loss of generality that $G$ is dense in $X$. Then $\upsilon  G=X$.
Since every $C$-embedded subspace of a regular Lindel\"{o}f is pseudo-$\omega_1$-compact \cite[Corollary~3.3]{AT1}, the space $G$ is pseudo-$\omega_1$-compact. It now follows from \cite[Corollary~8.3.3]{AT} that $X = \upsilon  G = \mu G$, i.e.~the Hewitt-Nachbin and Dieudonn\'{e} completions of $G$ coincide. By Corollary~\ref{C3.3}, $X$ is homeomorphic to a Lindel\"{o}f
topological group containing $G$ as a dense subgroup. Hence $X$ is $\mathbb{R}$-factorizable
by \cite[Theorem~8.1.6]{AT}. So $G$ is $\mathbb{R}$-factorizable as a $C$-embedded subgroup
of the $\mathbb{R}$-factorizable topological group $X$ (see \cite[Theorem~3.2]{HST}).
\end{proof}

\section{Completions of simply $sm$-factorizable paratopological groups}\label{Sec:4}
In this section we consider the Dieudonn\'e and Hewitt--Nachbin completions of simply
$sm$-factorizable \emph{paratopological} groups.

The following result follows from Lemmas~3 and~4 of \cite{ST1}:

\begin{lemma}\label{LL1}
Let $G$ be a Hausdorff paratopological group satisfying $Hs(G)\leq \omega$. Then the
$G_\delta$-closure of an arbitrary subgroup $H$ of $G$ is again a subgroup of $G$.
\end{lemma}

\begin{lemma}\label{LL}
The topological product $G=\prod_{\alpha\in A} G_\alpha$ of any family of strongly submetrizable paratopological groups satisfies $Hs(G)\leq \omega$.
\end{lemma}

\begin{proof}
According to \cite[Proposition~2.3]{T}, the class of paratopological groups with countable
Hausdorff number is closed under taking arbitrary products and subgroups. Therefore, it
suffices to show that any strongly submetrizable para\-topo\-logical group $H$ satisfies
$Hs(H)\leq \omega$. Since $H$ is strongly submetrizable, there exists a continuous
isomorphism $i\colon H\to H'$ onto a separable metrizable paratopological group $H'$.

Let $\{U_n: n\in \omega\}$ be a local base at the identity $e'$ in $H'$. Since $H'$ is metrizable, we
have the equality $\{e'\}=\bigcap_{n\in \omega}U_nU_n^{-1}$. Let $V_n=i^{-1}(U_n)$, $n\in\omega$.
For each neighborhood $V$ of the identity $e$ in $H$, we have that $\{e\}=\bigcap_{n\in \omega}V_nV_n^{-1}\subseteq V$. This implies that $Hs(H)\leq \omega$.
\end{proof}

\begin{theorem}\label{Th}
Let $G$ be a regular simply $sm$-factorizable paratopological group. Then the realcompactification
$\upsilon G$ of the space $G$ admits a natural structure of paratopological group containing $G$
as a dense subgroup and $\upsilon{G}$ is also simply $sm$-factorizable.
\end{theorem}

\begin{proof}
Since every regular paratopological group is a Tychonoff space \cite[Corollary~5]{BR}, so is $G$. Let
$\{f_\alpha: \alpha\in A\}$ be the family of continuous real-valued functions on $G$. Since $G$
is simply $sm$-factorizable,  it follows from Theorem~\ref{th1} that for each $\alpha\in A$,
there exist a continuous homomorphism $\pi_\alpha$ of $G$ onto a regular strongly submetrizable
paratopological group $H_\alpha$ and a continuous function $g_\alpha\colon H_\alpha\to \mathbb{R}$
such that $f_\alpha=g_\alpha\circ\pi_\alpha$. Since the family $\{f_\alpha: \alpha\in A\}$ separates
point and closed sets in $G,$ so does $\{\pi_\alpha: \alpha\in A\}$. Therefore, the diagonal product
of the family $\{\pi_\alpha: \alpha\in A\}$, denoted by $\pi$, is a topological isomorphism of $G$
onto the subgroup $\pi(G)\subseteq \Pi=\prod_{\alpha\in A} H_\alpha$. In what follows we identify
$G$ with the subgroup $\pi(G)$ of $\Pi$. Then the equality $f_\alpha=g_\alpha\circ\pi_\alpha$ acquires
the form $f_\alpha=g_\alpha\circ p_\alpha \res_{G}$, where $p_\alpha$ is the projection of $\Pi$ to
$H_\alpha$.

By Lemma~\ref{LL}, we have that $Hs(\Pi)\leq \omega$. Therefore, by Lemma~\ref{LL1}, the
$G_\delta$-closure of $G$ in $\Pi$, denoted by $H$, is a subgroup of $\Pi$. We claim that the subspace
 $H$ of $\Pi$ is realcompact. Indeed, for each $\alpha\in A$, $H_\alpha$ is a strongly submetrizable paratopological group, so $H_\alpha$ admits a coarser separable metrizable paratopological group topology. Hence the space $H_\alpha$ is Dieudonn\'{e} complete \cite[Proposition~6.10.8]{AT} and $|H_\alpha|\leq \cont$. In particular, the cellularity of $H_\alpha$ is less than or equal to $\cont$, which is Ulam non-measurable and, therefore, the space $H_\alpha$ is realcompact by \cite[Proposition~6.5.18]{AT}. Hence
 the space $\Pi=\prod_{\alpha\in A} H_\alpha$ is also realcompact. It also follows from the definition of
 $H$ that the complement $\Pi\setminus H$ is the union of family of $G_\delta$-sets in $\Pi$. Further,
 every $G_\delta$-set in $\Pi$ is the union of a family of zero-sets in $\Pi$, and the complement
 $\Pi\setminus Z$ is realcompact, for each zero-set $Z$ in $\Pi$ (see \cite[Corollary~3.11.8]{En}). By \cite[Corollary~3.11.7]{En}, the intersection of a family of realcompact subspaces of a space is realcompact. Therefore, $H$ is realcompact as the intersection of a family of cozero-sets in $\Pi$.

Let us show that $G$ is $C$-embedded in $H$, which implies that $\upsilon{G}= H$ (see
\cite[Theorem~8.6]{GJ}). Indeed, for each continuous real-valued function $f$ on $G$, there
exists $\alpha\in A$ such that $f=f_\alpha$. It follows from $f_\alpha=g_\alpha\circ p_\alpha
\res_{G}$ that $g_\alpha\circ p_\alpha \res_{H}\colon H\to \mathbb{R}$ is a continuous
extension of $f$ over $H$. We have thus proved that the realcompactification $\upsilon{G}$ of
$G$ admits a natural structure of a topological group containing $G$ as a dense subgroup.

It remains to verify that the paratopological group $H$ is a simply $sm$-factorizable. Take any continuous function $g\colon H\to \mathbb{R}$ and denote by $f$ the restriction of $g$ to $G$. Then we can find
$\alpha\in A$ such that $f=f_\alpha$, whence it follows that $f_\alpha=g_\alpha\circ p_\alpha\res_G$.
Since $\pi_\alpha(H)=H_\alpha$ and $H$ is Hausdorff, we have the equality $g=g_\alpha \circ p_\alpha \res_{H}$. Thus the continuous homomorphism $p_\alpha \res_H$ of $H$ to $H_\alpha$
factorizes the function $g.$ Therefore, by Theorem~\ref{th1}, $H$ is simply $sm$-factorizable.
\end{proof}

\begin{lemma}\label{Le:ast}
Let $G$ be a regular simply $sm$-factorizable paratopological group. Then the topological group
$G^\ast$ associated to $G$ satisfies $ib(G^\ast)\leq \cont$. Therefore, $c(G)\leq c(G^\ast)\leq 2^\cont$
and the equality $\upsilon{G}=\mu{G}$ is valid.
\end{lemma}

\begin{proof}
Let $U$ be an arbitrary neighborhood of the identity $e$ in the group $G^\ast$. It follows from the
definition of $G^\ast$ that there exists an open neighborhood $V$ of the identity in $G$ such that
$V\cap V^{-1}\subset U$. Every regular paratopological group is completely regular \cite[Corollary~5]{BR},
so there exists a cozero set $W$ in $G$ satisfying $e\in W\subset V$. Since $G$ is simply
$sm$-factorizable, we can find a continuous homomorphism $\pi\colon G\to H$ onto a separable
metrizable paratopological group $H$ such that $W=\pi^{-1}\pi(W)$. Hence $W^{-1}=\pi^{-1}\pi(W^{-1})$.
Combining the two equalities we see that
\begin{equation*}
W\cap W^{-1}=\pi^{-1}\pi(W\cap W^{-1}).
\end{equation*}

It is clear that $|H|\leq \cont$. Choose a subset $A$ of $G$ with $|A|\leq \cont$ such that
$\pi(A)=H$. Applying $(2)$ and taking into account that $e\in W\cap W^{-1}\neq\emptyset$ we deduce
that $A\cdot (W\cap W^{-1})=G=(W\cap W^{-1})\cdot A$. The latter equalities together with $e\in W\cap W^{-1}\subset V\cap V^{-1}\subset U$ imply that $A\cdot U=G=U\cdot A$. Hence the group $G^\ast$
satisfies $ib(G^\ast)\leq \cont$.

We now apply \cite[Theorem~5.4.10]{AT} to conclude that $c(G^\ast)\leq 2^{ib(G^\ast)}\leq 2^\cont$.
Since $G$ is a continuous image of $G^\ast$, we also have $c(G)\leq c(G^\ast)\leq 2^\cont$.
Finally, the cardinal number $2^\cont$ is Ulam non-measurable by \cite[Theorem~6.2.2]{AT}.
So the cardinality of every discrete family of open sets in $G$ is Ulam non-measurable and the
equality $\mu{G}=\upsilon{G}$ hods by \cite[Lemma~8.3.1]{AT}.
\end{proof}

\begin{corollary}\label{C3.8}
Let $G$ be a regular simply $sm$-factorizable paratopological group. Then the equality
$\mu{G}=\upsilon{G}$ hods. Furthermore, the space $\mu{G}$ admits a natural structure
of paratopological group containing $G$ as a dense subgroup and $\mu{G}$ is also simply
$sm$-factorizable.
\end{corollary}

\begin{proof}
Lemma~\ref{Le:ast} implies that the cardinality of every discrete family of open sets in $G$ is
at most $2^\cont$ and, hence, is Ulam non-measurable. So the equality $\mu{G}=\upsilon{G}$
hods by \cite[Lemma~8.3.1]{AT}. Therefore, both conclusions of the corollary follow directly from Theorem~\ref{Th}.
\end{proof}

\begin{corollary}
Let $G$ be a regular totally $\omega$-narrow and weakly Lindel\"{o}f paratopological group. Then
the equality $\mu{G}=\upsilon{G}$ hods. Furthermore, the Dieudonn\'{e} completion $\mu{G}$ of
the space $G$ admits a natural structure of paratopological group containing $G$ as a dense
subgroup and $\mu{G}$ is simply $sm$-factorizable.
\end{corollary}

\begin{proof}
The group $G$ is simply $sm$-factorizable, by Corollary~\ref{C2.8}. Hence the required
conclusions follow from Corollary~\ref{C3.8}.
\end{proof}

\begin{problem}
Does Lemma~\ref{Le:ast} remain valid without the assumption of the regularity of $G$?
\end{problem}

One can try to improve one of the conclusions of Lemma~\ref{Le:ast} as follows:

\begin{problem}\label{Prob:SM}
Does every Hausdorff (regular) simply $sm$-factorizable paratopological group $G$ satisfy
$c(G)\leq \cont$?
\end{problem}

It is worth mentioning that the \lq{Hausdorff\rq} and \lq{regular\rq} versions of Problem~\ref{Prob:SM}
are equivalent since every paratopological group $G$ satisfies $c(G)=c(G_{sr})$ (see
\cite[Proposition~2.2]{Tk15}) and the paratopological group $G_{sr}$ is regular provided $G$ is
Hausdorff.

The next result is close to Corollary~\ref{Cor:3.5}. In it, under stronger assumptions, we extend
some properties of topological groups to the wider class of paratopological groups.

\begin{corollary}\label{C3.10}
Let $G$ be a simply $sm$-factorizable paratopological group which is $C$-embedded in a
regular space $X$. If $X^2$ is Lindel\"of, then the group $G$ is $\mathbb{R}$-factorizable.
In addition, if $X$ is a Lindel\"{o}f $\Sigma$-space, then all subgroups of $G$ have
countable cellularity and for every family $\gamma$ of $G_\delta$-sets in $G$, the closure
of $\,\bigcup\gamma$ is a zero-set in $G$.
\end{corollary}

\begin{proof}
We can assume without loss of generality that $G$ is dense in $X$ because the class of
Lindel\"{o}f $\Sigma$-spaces is hereditary w.r.t.~taking closed subspaces. Then $\upsilon{G}=X$.
By Corollary~\ref{C3.8}, $X = \upsilon{G} = \mu{G}$ and $X$ is homeomorphic to a
paratopological group containing $G$ as a dense paratopological subgroup. Also, the
topological group $X^\ast$ associated to $X$ is topologically homeomorphic to a closed
subspace of the space $X^2$ \cite[Lemma~2.2]{AS} and, hence, $X^\ast$ is Lindel\"of.
Applying \cite[Theorem~3.6]{ST3} we see that $X$ is $\mathbb{R}$-factorizable. So $G$ is
$\mathbb{R}$-factorizable as a $C$-embedded subgroup of the $\mathbb{R}$-factorizable
paratopological group $X$ (see \cite[Theorem~3.2]{XY1}).

Assume that $X$ is a Lindel\"of $\Sigma$-space. Then $X^2$ is also a Lindel\"of $\Sigma$-space,
so $G$ is $\mathbb{R}$-factorizable. Let $G^\ast$ be the topological group associated to $G$. It
is clear that the identity embedding of $G$ to $X$ is topological isomorphism of $G^\ast$ onto
a subgroup of the topological group $X^\ast$. Since $X^\ast$ is homeomorphic to a closed
subspace of $X^2$, we deduce that $X^\ast$ is a Lindel\"of $\Sigma$-space. Every subgroup
$H$ of $X^\ast$ has countable cellularity by \cite[Corollary~5.3.21]{AT}. If $K$ is a subgroup of
$G$, then $H=j^{-1}(K)$ is a subgroup of both $G^\ast$ and $X^\ast$, where $j$ is the identity
mapping of $G^\ast$ onto $G$. It follows from the continuity of $j$ that $c(K)\leq c(H)\leq\omega$.

Finally, let $\gamma$ be a family of $G_\delta$-sets in $G$. Since $G$ is a Tychonoff
space, each element of $\gamma$ is the union of a family of zero-sets. Hence we can
assume that $\gamma$ consists of zero-sets in $G$. By our assumptions, $G$ is a dense
$C$-embedded subspace of $X$, so the closure in $X$ of every zero-set in $G$ is a
zero-set in $X$. Let $\gamma^\ast=\{\mathrm{cl}_X P: P\in\gamma\}$. Then $\gamma^\ast$
is a family of zero-sets in $X$ and \cite[Theorem~4.2]{ST2} implies that the closure  of
$\bigcup\gamma^\ast$ in $X$ is a zero-set. Since $\bigcup\gamma$ is dense in
$\bigcup\gamma^\ast$, we conclude that $\mathrm{cl}_G\bigcup\gamma$ is a
zero-set in $G$.
\end{proof}

\begin{remark}
In Corollary~\ref{C3.10}, the Lindel\"{o}fness of $X^2$ cannot be weakened to the Lindel\"{o}fness
of $X$. Indeed, the Sorgenfrey line $\mathbb{S}$ is a simply $sm$-factorizable paratopological
group by Theorem~\ref{th1}. Also, $\mathbb{S}$ is Lindel\"of but not $\mathbb{R}$-factorizable \cite[Remark~3.22]{XST}.
\end{remark}

\begin{corollary}[See Theorem~1 of \cite{ST1}]
Let $G$ be a regular $\mathbb{R}$-factorizable paratopological group. Then the Dieudonn\'{e}
completion of the space $G$, $\mu{G}$, admits a natural structure of paratopological group
containing $G$ as a dense subgroup and $\mu{G}$ is $\mathbb{R}$-factorizable.
\end{corollary}

\begin{proof}
Every regular paratopological group is a Tychonoff space \cite[Corollary~5]{BR}. Since
$G$ is $\mathbb{R}$-factorizable (hence, simply $sm$-factorizable), we can apply
Corollary~\ref{C3.8} to conclude that both multiplication and inversion on $G$ admit
continuous extensions to $\mu{G}$ making the latter space into a paratopological group.
The $\mathbb{R}$-factorizability of the group $\mu{G}$ can be deduced as in \cite{ST1}.
\end{proof}

Our proof of Theorem~\ref{P3.11} below requires the next simple fact:

\begin{lemma}\label{Le:NS}
Let $\mathcal{N}$ be a family of subgroups of a paratopological group $G$ such that
every neighborhood of the identity $e$ in $G$ contains an element of $\mathcal{N}$.
Then every neighborhood of $e$ in the topological group $G^\ast$ associated to $G$
also contains an element of $\mathcal{N}$.
\end{lemma}

\begin{proof}
Take an arbitrary neighborhood $U$ of $e$ in $G^\ast$. It follows from the definition of the
topology of $G^\ast$ that there exists an open neighborhood $V$ of $e$ in $G$ such that
$V\cap V^{-1}\subset U$. By the assumptions of the lemma, there exists $N\in\mathcal{N}$
satisfying $N\subset V$. Since $N$ is a subgroup of $G$, we have that $N=N^{-1}
\subset V^{-1}$. Therefore, $N\subset V\cap V^{-1}\subset U$.
\end{proof}

\begin{theorem}\label{P3.11}
Let $G$ be a regular paratopological group. If the topological group $G^\ast$ associated to $G$
is simply $sm$-factorizable and $\omega$-narrow, then so is $G$.
\end{theorem}

\begin{proof}
Let $\mathcal {N}$ be the family of closed invariant subgroups $N$ of $G$ such that the quotient
paratopological group $G/N$ is strongly submetrizable.\smallskip

{\bf Claim 1.} \emph{The family $\mathcal {N}$ is closed under countable intersections.}\smallskip

Take any countable subfamily $\mathcal {C}=\{N_{k}: k\in\omega\}$ of $\mathcal {N}$. For every
$k\in\omega$, let $\varphi_k\colon G\to G/N_k$ be the quotient homomorphism. Then the diagonal
product $\psi\colon G\to \prod_{k\in \omega} G/N_k$ of the family $\{\varphi_k: k\in \omega\}$ is
a continuous homomorphism. Since every group $G/N_k$ is strongly submetrizable, so is the
countable product $\prod_{k\in \omega} G/N_k$. Further, the subgroup $\psi(G)$ of $\prod_{k\in\omega}
G/N_k$ is also strongly submetrizable. Clearly, the kernel of $\psi$ satisfies $\ker\psi=\bigcap\mathcal {C}$.
It is easy to see that the quotient group $G/\ker\psi$ is  strongly submetrizable and the latter implies
that $\bigcap\mathcal {C}=\ker\psi\in \mathcal {N}$. This proves Claim~1.\smallskip

Let $\mathcal{N}=\{N_i: i\in I\}$ and $\varphi_i\colon G\to G/N_i$ be the quotient
homomorphism, where $i\in I$.\smallskip

{\bf Claim 2.} \emph{The diagonal product $\varphi\colon G\to \varphi(G)\subseteq \Pi=\prod_{i\in I}
G/N_i$ of the family $\{\varphi_i: i\in I\}$ is a topological isomorphism. In particular, every neighborhood
of the identity in $G$ contains an element of $\mathcal{N}$.}\smallskip

Since $G$ is regular and totally $\omega$-narrow, it is $\omega$-balanced \cite[Proposition~3.8]{ST2}
and satisfies $Ir(G)\leq \omega$ \cite[Theorem~2]{Sa}. Therefore, it follows from Theorem~3.6 and Lemma~3.7 of \cite{T}  that for each open neighborhood $U$ of the identity $e$ in $G$, one can find
a continuous homomorphism $p\colon G\to H$ onto a separable metrizable paratopological group $H$
and an open neighborhood $V$ of the identity in $H$ such that $p^{-1}(V)\subseteq U$. Hence the
kernel of $p$ belongs to $\mathcal {N}$ and satisfies $\ker{p}\subset U$. This implies that the
family $\{\varphi_i: i\in I\}$ separates the points and closed sets in $G$ and, therefore, the diagonal
product of this family, say, $\varphi$ is a topological isomorphism of $G$ onto $\varphi(G)$, as claimed.
\smallskip

{\bf Claim 3.} \emph{Every $G_\delta$-set $P$ in $G^\ast$ with $e\in P$ contains some
$N\in\mathcal{N}$.}\smallskip

Take a $G_\delta$-set $P$ in $G^\ast$ containing the identity $e$. Let $\{U_k: k\in\omega\}$
be a family of neighborhoods of the identity in $G^\ast$ such that $P=\bigcap_{k\in\omega} U_k$.
It follows from Lemma~\ref{Le:NS} and Claim~2 that for every $k\in\omega$, there exists
$N_k\in\mathcal{N}$ such that $N_k\subset U_k$. Then by Claim~1, $N=\bigcap_{k\in\omega} N_k$
is an element of $\mathcal{N}$ satisfying $N\subset \bigcap_{k\in\omega} U_k=P$.\smallskip


Take any continuous function $f\colon G\to \mathbb{R}$. Then $f$ remains continuous on $G^\ast$.
Since $G^\ast$ is simply $sm$-factorizable, it follows from item (2) of Theorem~\ref{th1} that we
can find a continuous homomorphism $\pi$ of $G^\ast$ onto a strongly submetrizable topological
group $H$ and a continuous function $g\colon H\to \mathbb{R}$ such that $f=g\circ \pi$. Since $H$
is strongly submetrizable, the identity $e_H$ of $H$ is a $G_\delta$-set in $H$. Therefore, $\ker\pi$
is a $G_\delta$-set in $G^\ast$. Then Claim~3 implies that there is an $N_i\in \mathcal {N}$
such that $N_i\subseteq \ker\pi$. Since $f$ is constant on each coset $x\cdot \ker\pi$ and the
groups $G$ and $G^\ast$ share the same underlying set, we see that $f$ is constant on $xN_i$
for each $x\in G$. Therefore, there exists a function $h\colon G/N_i\to \mathbb{R}$ such that
$f=h\circ \varphi_i$, where $G/N_i$ is endowed with the quotient topology and $\varphi_i$ is the
quotient homomorphism of $G$ onto $G/N_i$. Since $\varphi_i$ is continuous and open, $h$ is
continuous. Clearly, $G/N_i$ is strongly submetrizable, so $G$ is simply $sm$-factorizable by Theorem~\ref{th1}. Finally, $G$ is $\omega$-narrow as a continuous homomorphic image of
$G^\ast$.
\end{proof}

Theorem~\ref{P3.11} makes it natural to ask the following questions:

\begin{problem}
Let $G$ be a regular simply $sm$-factorizable paratopological group. Is  the topological group
$G^\ast$ associated to $G$ simply $sm$-factorizable? What if $G$ is a regular $\mathbb{R}$-factorizable paratopological group?
\end{problem}

\begin{question}
Can the requirement of $\omega$-narrowness of $G^\ast$ be dropped in Theorem~\ref{P3.11}?
Does the $\omega$-narrowness of $G$ suffice?
\end{question}

In Example~\ref{Ex:NN} we answer both parts of the above question in the negative. First we
present an auxiliary lemma in which $\mathbb{Z}$ stands for the discrete additive group
of integers.

\begin{lemma}\label{Le:Aux}
There exists a countable dense subgroup $S$ of the product $\mathbb{Z}^{\cont}$
such that for every $x\in S$, every finite set $C\subset\cont$ and every $k\in\mathbb{Z}$,
one can find $s\in S$ satisfying $s(\alpha)\leq x(\alpha)$ for each $\alpha\in\cont\setminus C$
and $s(\alpha)\leq k$ for each $\alpha\in C$.
\end{lemma}

\begin{proof}
Our argument is a modification of the proof of the Hewitt--Marczewski-Pondiczery theorem
as presented in \cite[Theorem~2.3.15]{En}. Let $\mu$ be a separable metrizable topology
on the index set $\cont$. Denote by $\mathcal{A}$ a countable base for $(\cont,\mu)$. Take
a countable dense subset $R_0$ of the space $\Pi=\mathbb{Z}^\cont$ and let $S_0=\hull{R_0}$.
Assume that we have defined countable subgroups $S_0\subset \cdots\subset S_n$ of $\Pi$.
For every $x\in S_n$, every finite disjoint subfamily $\nu$ of $\mathcal{A}$ and every
$k\in\mathbb{Z}$, we define an element $y_{x,\nu,k}\in\Pi$ by the rule
\begin{equation*}
y_{x,\nu,k}(\alpha)=\begin{cases} x(\alpha)&\text{if $\alpha\in\cont\setminus \bigcup\nu$};\\
\min\{k,x(\alpha)\}&\text{if $\alpha\in \bigcup\nu$}.
\end{cases}
\end{equation*}
Let $R_{n+1}=\{y_{x,\nu,k}: x\in S_n,\ \nu\in [\mathcal{A}]^{<\omega},\ k\in\mathbb{Z}\}$.
Clearly $R_{n+1}$ is countable, so $S_{n+1}=S_n+\hull{R_{n+1}}$ is countable as well.

We claim that the subgroup $S=\bigcup_{n\in\omega} S_n$ of $\Pi$ is as required.
Note that $S$ is countable and dense in $\Pi$ since $S_0\subset S$. Take an element
$x\in S$, a finite subset $C=\{\alpha_1,\ldots,\alpha_r\}$ of $\cont$ and an integer $k$.
Then $x\in S_n$ for some $n\in\omega$. Since the space $(\cont,\mu)$ is Hausdorff,
we can find pairwise disjoint elements $U_1,\ldots,U_r$ of $\mathcal{A}$ such that
$\alpha_i\in U_i$ for each $i\leq r$. Let $\nu=\{U_1,\ldots,U_r\}$. Then the point
$s=y_{x,\nu,k}\in R_{n+1}\subset S_{n+1}\subset S$ satisfies $s(\alpha)\leq x(\alpha)$
for each $\alpha\in\cont\setminus C$ and $s(\alpha_i)\leq k$ for each $i\leq r$. This
completes the proof.
\end{proof}

\begin{example}\label{Ex:NN}
\emph{There exists a regular $\omega$-narrow paratopological Abelian group $G$ such that
the topological group $G^\ast$ associated to $G$ is discrete with $|G^*|\leq\cont$ (hence simply $sm$-factorizable by \cite[Proposition 5.15]{AT1}),
but $G$ fails to be simply $sm$-factorizable.}
\end{example}

\begin{proof}
We modify the construction described in \cite[Example~2.9]{T}. In fact, our group $G$
will be a (dense) subgroup of the paratopological group constructed there.

Let $\mathbb{Z}$ be the discrete additive group of the integers and $\Pi=\mathbb{Z}^\cont$
be the product of $\cont=2^\omega$ copies of $\mathbb{Z}$. For every $x\in\Pi$, let
$$
\supp(x)=\{\alpha\in \cont: x(\alpha)\neq 0\}.
$$
Then $\sigma=\{x\in\Pi: |\supp(x)|<\omega\}$ is a subgroup of $\Pi$ which is called the
\emph{$\sigma$-product} of $\cont$ copies of $\mathbb{Z}$. It is clear that $|\sigma|=\cont$.
Let $S$ be a countable dense subgroup of $\Pi$ as in Lemma~\ref{Le:Aux} ($\Pi$ carries
the Tychonoff product topology). Clearly $G=\sigma+S$ is a subgroup of $\Pi$.

For every $A\subset \cont$, we define a subset $U_A$ of $G$ by
$$
U_A = \{x\in G: x(\alpha)=0 \mbox{ for each } \alpha\in A \mbox{ and } x(\alpha)\geq 0
\mbox{ for each } \alpha\in\cont\}.
$$
Note that each $U_A$ is a \emph{subsemigroup} of $G$, i.e.~$U_A+U_A\subset U_A$.
Also, each $U_A$ contains the identity element of $G$.  These properties of the sets
$U_A$ imply that the family
$$
\mathcal{B} = \{x+U_A: x\in G,\ A\subset \cont,\ |A|<\omega\}
$$
is a base for a paratopological group topology $\tau$ on $G$ and the sets $U_A$, with a
finite set $A\subset\cont$, is a local base at the identity of the paratopological group $(G,\tau)$.
It is clear that the topology $\tau$ is (strictly) finer than the topology of $G$ inherited from
the Tychonoff product $\mathbb{Z}^\cont$, so the space $(G,\tau)$ is Hausdorff.\smallskip

\noindent
{\bf Claim~1.} \emph{The group $(G,\tau)$ is $\omega$-narrow.}
\smallskip

Consider a basic open neighborhood $U_A$ of the identity in $G$, where $A$ is finite.
Denote by $\sigma_A$ the set of all $x\in\sigma$ such that $\supp(x)\subset A$. It is
clear that $\sigma_A$ is a countable subgroup of $\sigma$. We claim that $G=S+\sigma_A+U_A$.
Indeed, let $y\in G$ be an arbitrary element. Then $y=x+a$ for some $x\in S$ and $a\in\sigma$. If
$a$ is the identity element of $G$ (equivalently, $\supp(a)=\emptyset$), then $y=x\in S\subset
S+U_A$. Otherwise let $C=\supp(a)\setminus A$ and $k=\min\{y(\alpha): \alpha\in\supp(a)\}$.
According to our choice of $S$, there exists an element $s\in S$ such that $s(\alpha)\leq
x(\alpha)$ for each $\alpha\in\cont\setminus C$ and $s(\alpha)\leq k$ for each $\alpha\in C$.
Since $x$ and $y$ coincide on $\cont\setminus \supp(a)$, our definition of $k$ implies that
$s(\alpha)\leq y(\alpha)$ for each $\alpha\in\cont$. Choose an element $b\in\sigma_A$
such that $y(\alpha)=s(\alpha)+b(\alpha)$ for each $\alpha\in A$. Then $y\in s+b+U_A
\subset S+\sigma_A+U_A$. This proves the equality $G=S+\sigma_A+U_A$. Since the
subset $S+\sigma_A$ of $G$ is countable, we conclude that the group $(G,\tau)$ is
$\omega$-narrow. This proves Claim~1.\smallskip

\noindent
{\bf Claim~2.} \emph{The set $U_A$ is clopen in $(G,\tau)$, for each finite subset $A$
of $\cont$.}\smallskip

Indeed, if $x\in G\setminus U_A$, then either $x(\alpha)\neq 0$  for some $\alpha\in A$
or $x(\beta)<0$ for some $\beta\in\cont$. In the first case, $x+U_B$ is an open neighborhood
of $x$ disjoint from $U_A$, where $B=\{\alpha\}$. In the second case, $x+U_B$ is an open
neighborhood of $x$ disjoint from $U_A$, where $B=\{\beta\}$. Therefore, the complement
$G\setminus U_A$ is open in $(G,\tau)$ and $U_A$ is clopen. This proves Claim~2.

It follows from Claim~2 that the base $\mathcal{B}$ of $(G,\tau)$ consists of clopen sets
and, hence, this space is zero-dimensional. In particular, $(G,\tau)$ is regular.

Let $O=U_\emptyset$.  Then $O\cap (-O)=\{\bar{0}\}$, where $\bar{0}$ is the
identity element of $G$. Hence the topological group $(G,\tau)^\ast$ associated to
$(G,\tau)$ is discrete.\smallskip

\noindent
{\bf Claim~3.} \emph{The group $(G,\tau)$ is not simply $sm$-factorizable.}\smallskip

By Claim~2, the set $O$ is clopen in $(G,\tau)$. Let $f$ be the characteristic function of
$O$, i.e.~$f(x)=1$ if $x\in O$ and $f(x)=0$ otherwise. Then $f$ is continuous. Suppose
for a contradiction that there exists a continuous homomorphism $\varphi\colon (G,\tau)\to H$
to a second-countable paratopological group $H$ such that $O=\varphi^{-1}\varphi(O)$.
Let $\{V_n: n\in\omega\}$ be a local base at the identity of $H$. For every $n\in\omega$,
take a finite subset $A_n$ of $\cont$ such that $\varphi(U_{A_n})\subset V_n$. Then the
set $B=\bigcup_{n\in\omega} A_n$ is countable and $U_B\subset\ker\varphi$. Since
$\ker\varphi$ is a subgroup of $G$, we see that the subgroup $\hull{U_B}$ of $G$
generated by $U_B$ is contained in $\ker\varphi$. An easy verification shows that
$$
\hull{U_B}=\{x\in G: x(\alpha)=0 \mbox{ for each } \alpha\in B\}.
$$
Denote by $p_B$ the natural projection of $G$ to $\mathbb{Z}^B$. Then $\ker p_B=\hull{U_B}
\subset \ker\varphi$, so our choice of $\varphi$ implies that $O=p_B^{-1}p_B(O)$, which is
clearly false. Indeed, take an arbitrary element $x\in G$ such that $x(\alpha)=0$ for each
$\alpha\in B$ and $x(\beta)<0$ for some $\beta\in\cont\setminus B$. Then $x\notin O$,
while $x\in p_B ^{-1}p_B(O)$. This contradiction proves that the group $(G,\tau)$ fails to
be simply $sm$-factorizable.
\end{proof}

Since every Hausdorff $\mathbb{R}$-factorizable topological group is $\omega$-narrow and
simply $sm$-factorizable, the next corollary is immediate from Theorem~\ref{P3.11}.

\begin{corollary}\label{Cor:4.11}
If the topological group $G^\ast$ associated to a regular paratopological group $G$ is
$\mathbb{R}$-factorizable, then $G$ is simply $sm$-factorizable.
\end{corollary}

\begin{problem}\label{Prob:Reg}
Can one weaken the regularity of $G$ in Corollary~\ref{Cor:4.11} to the
Hausdorff separation property?
\end{problem}

\begin{theorem}\label{Th3.14}
Let $G$ be a regular paratopological group such that the topological group $G^\ast$ associated
to $G$ is $\omega$-narrow and simply $sm$-factorizable. Then the realcompactification $\upsilon{G}$
of the space $G$ admits a natural structure of paratopological group containing $G$ as a dense
subgroup and the group $\upsilon{G}$ is simply $sm$-factorizable.
\end{theorem}

\begin{proof}
According to Theorem~\ref{P3.11}, $G$ is simply $sm$-factorizable. Hence Corollary~\ref{C3.8}
implies that the space $\upsilon{G}$ admits the structure of paratopological group containing $G$
as a dense subgroup. By Theorem~\ref{Th}, the group $\upsilon{G}$ is simply $sm$-factorizable.
\end{proof}

It is well known that every Hausdorff $\mathbb{R}$-factorizable topological group is $\omega$-narrow
and simply $sm$-factorizable (see \cite[Proposition~8.1.3]{AT} and \cite[Theorem~5.9]{AT1}). Thus
the next corollary follows from Theorem~\ref{Th3.14}.

\begin{corollary}[See Theorem~2 of \cite{ST1}]\label{C3.15}
Let $G$ be a regular paratopological group such that the topological group $G^\ast$ associated to
$G$ is $\mathbb{R}$-factorizable. Then the realcompactification $\upsilon{G}$ of the space $G$
admits a natural structure of paratopological group containing $G$ as a dense subgroup and the
equality $\upsilon{G}=\mu{G}$ holds.
\end{corollary}

\begin{remark}
Let $G$ be as in Corollary~\ref{C3.15}. It is shown in \cite[Theorem~2]{ST1} that the
topological groups $(\upsilon{G})^\ast$ and $\upsilon(G^\ast)$ are topologically isomorphic
and $\mathbb{R}$-factorizable. However, we do not know whether the paratopological group
$G$ is $\mathbb{R}$-factorizable (see \cite[Problem~5.1]{ST3}). Notice that $G$ is simply
$sm$-factorizable, by Theorem~\ref{Th3.14}.
\end{remark}

{\bf Acknowledgement:} This paper is dedicated to Professor Lin Shou on the occasion of his 60th anniversary. He is a distinguished teacher and is one of the founders of the Chinese school of Generalized Metric Spaces Theory. His deep mathematical insight and his warm and sincere personality greatly influenced us.



\vskip0.9cm


\begin{thebibliography}{99}

\bibitem{AS} O.T.~Alas, M.~Sanchis,
\newblock Countably compact paratopological groups,
\newblock \textit{Semigroup Forum} \textbf{74}, no.~3 (2007),
423--438.

\bibitem{AT} A.V.~Arhangel'skii, M.~Tkachenko,
\newblock \emph{Topological Groups and Related Structures},
\newblock Atlantis Studies in Mathematics, vol.~1, Atlantis Press, Paris;
World Scientific Publishing Co. Pte. Ltd., Hackensack, NJ, 2008, 781~pp.

\bibitem{AT1} A.V.~Arhangel'skii, M.~Tkachenko,
\newblock $C$-extensions of topological groups,
\newblock Topol. Appl. \textbf{235} (2018), 54--72.

\bibitem{BR} T. Banakh, A. Ravsky,
\newblock Each regular paratopological group is completely regular,
\newblock Proc. Am. Math. Soc. \textbf{145} (3) (2017), 1373--1382.

\bibitem{CR} W.W. Comfort, K. A. Ross,
\newblock Pseudocompactness and uniform continuity in topological groups,
\newblock Pacific J. Math. \textbf{16} (1966), 483--496.


\bibitem{En} R. Engelking,
\newblock \emph{General Topology}, Heldermann, Berlin (1989).

\bibitem{GJ} L. Gillman, M. Jerison,
\newblock \emph{Rings of continuous functions},
\newblock Springer-Verlag, Berlin 1976.

\bibitem{HST} S. Hern\'andez, M. Sanchis, M. Tkachenko,
\newblock Bounded sets in spaces and topological groups,
\newblock Topol. Appl. \textbf{101} (2000), 21--43.

\bibitem{HR} E.~Hewitt and K.~Ross,
\newblock \textit{Abstract Harmonic Analysis}, Volume I,
\newblock Springer-Verlag, Berlin-G\"ottingen-Heidelberg 1979.


\bibitem{PZ} L.X. Peng, P. Zhang,
\newblock $\mathbb{R}$-factorizable, simply $sm$-factorizable paratopological
groups and their quotients,
\newblock Topol. Appl. \textbf{258} (2019), 378--391.

\bibitem{Pes} V.\,G.~Pestov,
\newblock On embeddings and condensations of topological groups,
\newblock Math. Notes \textbf{31} (1982), 228--230.

\bibitem{Rav01} O.\,V.~Ravsky,
\newblock Paratopological groups II,
\newblock Mat. Studii \textbf{17} (2002), 93--101.

\bibitem{RS} E.\,A.~Reznichenko, and O.\,V.~Sipacheva,
\newblock The free topological group on the Sorgenfrey line is not
$\mathbb{R}$-factorizable,
\newblock \textit{Topol. Appl.} \textbf{160} no.~11 (2013), 1184--1187.

\bibitem{Sa} I.~Sanchis,
\newblock Cardinal invariants of paratopological groups,
\newblock Topol. Algebra Appl. \textbf{1} (2013), 37--45. DOI: 102478/taa-2013-0005.

\bibitem{ST2} M. Sanchis, M. Tkachenko,
\newblock Totally Lindel\"{o}f and totally $\omega$-narrow paratopological
groups,
\newblock Topol. Appl. \textbf{155} (2008), 322--334.

\bibitem{ST3} M.~Sanchis, M.~Tkachenko,
\newblock $\mathbb{R}$-factorizable paratopological groups,
\newblock Topol. Appl. 157 (2010), 800--808.

\bibitem{ST} M. Sanchis, M. Tkachenko,
\newblock Dieudonn\'{e} completion and $PT$-groups,
\newblock Applied Categorical Structures \textbf{20} (1) 2012), 1--20.

\bibitem{ST1} M. Sanchis, M. Tkachenko,
\newblock Completions of paratopological groups and bounded sets,
\newblock Monatsh. Math. \textbf{183} (2017), 699--721.


\bibitem{Tk89} M.~Tkachenko,
\newblock Generalization of a theorem of Comfort and Ross,
\newblock Ukrainian Math. J. \textbf{41} (1989), 334--338.
\newblock Russian original in: Ukrain. Mat. Zh. \textbf{41} (1989),
377--382.

\bibitem{T} M.~Tkachenko,
\newblock Embedding paratopological groups into topological products,
\newblock Topol. Appl. \textbf{156} (2009), 1298--1305.

\bibitem{T2} M. Tkachenko,
\newblock Axioms of separation in semitopological groups and related functors,
\newblock Topol. Appl. \textbf{161} (2014), 364--376.

\bibitem{Tk15} M. Tkachenko,
\newblock Applications of the reflection functors in paratopological groups,
\newblock Topol. Appl. \textbf{192} (2015), 176--187.

\bibitem{T1} M. Tkachenko,
\newblock Pseudocompact Topological groups.
\newblock Section~2 in: M.~Hru\v{s}\'{a}k, \'{A}. Tamariz-Mascar\'{u}a, M. Tkachenko (Eds.),
\newblock \emph{Pseudocompact Topological Spaces.} Developments in Mathematics,
vol.~55.  Springer, Cham, 2018.

\bibitem{XST} L.H. Xie, S. Lin, M. Tkachenko,
\newblock Factorization properties of paratopological groups,
\newblock Topol. Appl. \textbf{160} (2013), 1902--1917.

\bibitem{XY} L.H. Xie, P.F. Yan,
\newblock A note on bounded sets and $C$-compact sets in paratopological groups,
\newblock Topol. Appl. \textbf{265} 2019, 106834.

\bibitem{XY1}L.H. Xie, P.F. Yan,
\newblock The continuous $d$-open images and subgroups of $\mathbb{R}$-factorizable
paratopological groups,
\newblock arXiv preprint, arXiv:1905.09577, 2019.


\end{thebibliography}
\end{document}